\documentclass{amsart}
\usepackage{amsfonts,amssymb,amsmath,amsthm}
\usepackage{url}
\usepackage{enumerate}

\newtheorem{thrm}{Theorem}[section]
\newtheorem{lem}[thrm]{Lemma}
\newtheorem{prop}[thrm]{Proposition}
\newtheorem{cor}[thrm]{Corollary}
\newtheorem*{ack}{Acknowledgements}
\theoremstyle{definition}

\newtheorem{remark}[thrm]{Remark}
\numberwithin{equation}{section}
\renewcommand{\o}{\mathfrak{O}}

\newcommand{\w}{\varpi}
\renewcommand{\d}{\delta}
\newcommand{\e}{\epsilon}
\newcommand{\g}{\gamma}

\renewcommand{\l}{\lambda}
\newcommand{\A}{\mathbb{A}}
\newcommand{\C}{\mathbb{C}}

\newcommand{\N}{\mathbb{N}}
\newcommand{\Z}{\mathbb{Z}}
\newcommand{\1}{\mathbf{1}}

\newcommand{\sm}{\mathcal{C}^\infty}

\author{N. Matringe}
\address{Universit\'e de Poitiers, Laboratoire de Math\'ematiques et Applications,\\
T\'el\'eport 2 - BP 30179, Boulevard Marie et Pierre Curie,\\
 86962, Futuroscope Chasseneuil Cedex}
\email{matringe@math.univ-poitiers.fr}
\thanks{ }

\begin{document}

\title[A specialisation of the Bump-Friedberg $L$-function]{A specialisation of the Bump-Friedberg $L$-function}

\begin{abstract}
We study the restriction of the Bump-Friedberg integrals to affine lines $\{(s+\alpha,2s),s\in\C\}$. 
It has a simple theory, very close to that of the Asai $L$-function. It is an integral representation of the product
 $L(s+\alpha,\pi)L(2s,\Lambda^2,\pi)$ which we denote by $L^{lin}(s,\pi,\alpha)$ for this abstract, when $\pi$ is a cuspidal automorphic 
representation of $GL(k,\A)$ for 
$\A$ the adeles of a number field. When $k$ is even, we show that for a cuspidal automorphic representation $\pi$, 
the partial $L$-function $L^{lin,S}(s,\pi,\alpha)$ has a pole at $1/2$, if and only if $\pi$ admits a (twisted) global period, 
this gives a more direct proof of a 
theorem of Jacquet and Friedberg, asserting 
that $\pi$ has a twisted global period if and only if $L(\alpha+1/2,\pi)\neq 0$ and $L(1,\Lambda^2,\pi)=\infty$.
When $k$ is odd, the partial $L$-function is holmorphic in a neighbourhood of $Re(s)\geq 1/2$ when $Re(\alpha)$ is 
$\geq 0$.
\end{abstract}
\maketitle

\section{Introduction}

In this paper, we study the restriction of the integrals of two complex variables $(s_1,s_2)$ defined in \cite{BF}, and attached to global
 and local smooth complex representations of $GL(2n)$, to the line $s_2=2(s_1-\alpha)$, for $\alpha \in \C$. We actually study 
slightly more general integrals. It turns out that these integrals have a theory very close to the theory of Asai $L$-functions, 
whose Rankin-Selberg theory, initiated by Flicker, is quite complete now (see \cite{F}, \cite{FZ}, \cite{K}, \cite{AKT}, \cite{AR}, \cite{M1}, 
\cite{M2}, \cite{M3}).\\ 
In \cite{BF}, for $\pi$ a cuspidal automorphic representation of $GL(n)$ of the adeles $\A$ of a global field, the authors mainly define the global 
integrals as the integral of a cusp form in $\pi$ against an Eisenstein series, prove their functional equation, and show that they unravel to the 
integral of the Whittaker function associated to the cusp form against a function in the space of an induced representation. This allows them to obtain
 an Euler factorisation, they then compute the local integrals at the unramified places and thus obtain an integral representation of 
$L(s_1,\pi)L(s_2,\Lambda^2,\pi)$. The location of the possible poles is quickly discussed.\\
In the first paragraphs of the Section \ref{local}, we define the $L$ function $L^{lin}(s,\pi)$ for a generic representation 
$\pi$ of $GL(n,F)$ for $n$ even equal to $2m$ (Theorem \ref{nonarchL}), when $F$ is a non-archimedean local field, and show a nonvanishing result. A much more 
complete study of this non archimedean $L$-function can be found in \cite{M4}.\\ 
We compute the Rankin-Selberg integrals when $\pi$ is unramified in Section \ref{unramified}.\\
In the archimedean case (Section \ref{archi}), we prove results of convergence and non-vanishing of the archimedean integrals, 
that we use in the global situation.\\
Section \ref{global} is devoted to the global theory. We take $\pi$ a smooth cuspidal automorphic representation of $GL(2m,\A)$, 
for $\A$ the adele ring of a number field $k$. We first study the integrals $I(s,\phi,\Phi)$ associated to a cusp form $\phi$ in the space of 
$\pi$, a Schwartz function $\Phi$ on $\A^m$, and a character $\chi$ of $GL(m,\A)\times GL(m,\A)$ trivial on the diagonal embedding 
of $Gl(m,\A)$, using mirabolic Eisenstein series similar to those in \cite{JS}, \cite{JS1}, or \cite{F}, this 
seems to avoid the normalisation by the $L$-function of the central character of $\pi$ as in \cite{BF}. We thus obtain their meromorphicity, 
functional equation as well as the location of their possible poles in Theorem \ref{I}. Then we prove 
the equality of these integrals (Theorem \ref{egal}) with the Rankin-Selberg integrals $\Psi(s,W_{\phi},\chi,\Phi)$ obtained by 
integrating the Whittaker functions associated to $\phi$, and thus get the Euler factorisation in Section \ref{unfolding}. The proof of this 
is similar to that of \cite{BF}, but we use successive partial Fourier expansions (Proposition \ref{Fourier}), which makes 
the computations quicker.\\ 
In the last part, we define the partial $L$-function $L^{lin,S}(s,\pi,\chi)$, and show that it is meromorphic, and moreover, when 
the real part of the idele class character defining $\chi$ is non negative, that it is 
holomorphic for $Re(s)>1/2$, and that it has a pole at $1/2$ if and only if $\pi$ has a twisted global period (Theorem \ref{partialpole}). 
We deduce from this the theorem of Friedberg and Jacquet discussed in the abstract (Theorem \ref{FJ}). It seems that this proof of the aforementioned theorem is 
not in \cite{BF} because the local $L$ functions were not really studied in [loc. cit.]. Especially the nonvanishing results, which are 
however easy (especially in the nonarchimedean case), are absent in \cite{BF}. Studying the Bump-Friedberg $L$-function 
through its restriction to complex lines of slope $2$ (in particular considering it as a function of one complex variable) simplifies 
the analysis.\\
In Section \ref{odd}, we give the results for the odd case. The global Rankin-Selberg integrals are holomorphic this time, and we prove that 
the partial $L$-function is holomorphic in a neighbourhood of $Re(s)\geq 1/2$ with the same assumption on the idele 
character defining $\chi$.\\

{\ack{I thank the referee for his careful reading and valuable suggestions. I thank the CNRS for giving me 
a "d\'el\'egation" in 2012 when this work was started. This work was partially supported by the research project ANR-13-BS01
-0012 FERPLAY.}}

\section{Preliminaries}\label{prelim}

Let $n$ belong to $\N$. We will use the notations $G_n$ for the algebraic group $GL(n)$, $Z_n$ for its center, $P_n$ for its mirabolic subgroup (the matrices in $G_n$ with last row $(0,\dots,0,1)$), $B_n$ for the Borel subgroup of upper triangular matrices in $G_n$, $N_n$ for its unipotent radical. We will write $U_n$ for the unipotent radical of $P_n$, and $u(x)$ will be the matrix $\begin{pmatrix} I_{n-1} & x \\ & 1\end{pmatrix}$ in $U_n$. 
Let $\mathcal M_k$ denote the set of $k\times k$ square matrices and $\mathcal M_{a,b}$ the $a\times b$ matrices. 
For $n> 1$, the map $g\mapsto \begin{pmatrix} g & \\ & 1 \end{pmatrix}$ is an embedding of the group $G_{n-1}$ in $G_{n}$. We denote by $P_n$ the subgroup $G_{n-1}U_n$ of $G_n$. This is the mirabolic subgroup of $G_n$.\\
Suppose $n=2m$ is even.  Let $w_n\in G_n$ be the permutation matrix for the permutation given by
\[
\begin{pmatrix} 1 & 2 & \cdots & m & | & m+1 & m+2 & \cdots & 2m \\ 
                1 & 3 & \cdots & 2m-1 & | & 2 & 4 & \cdots & 2m
                \end{pmatrix}.
\]
In this case we denote by $M_n$ the standard Levi of $G_n$ associated to the partition $(m,m)$ of $n$. Let $H_n=w_n M_n w_n^{-1}$, we 
write $h(g_1,g_2)=w_n diag(g_1,g_2)w_n^{-1}$ for $diag(g_1,g_2)$ in $M_n$.

Suppose $n=2m+1$ is odd. In this case we let $w_n$ be the permutation matrix in $G_n$ associated to the permutation
\[
\begin{pmatrix} 1 & 2 & \cdots & m & | & m+1 & m+2 & \cdots &  2m & 2m+1 \\ 
                1 & 3 & \cdots & 2m-1 & | & 2 & 4 & \cdots &  2m & 2m+1
                \end{pmatrix}.
\]
so that $w_{2m}=w_{2m+1}|_{G_{2m}}$ and let $w_{2m+1}=w_{2m+2}|_{GL_{2m+1}}$ so that $w_{2m+1}$ is the permutation matrix corresponding to 
\[
\begin{pmatrix} 1 & 2 & \cdots & m+1 & | & m+3 & m+4 & \cdots &  2m+1 \\ 
                1 & 3 & \cdots & 2m+1 & | & 2 & 4 & \cdots &  2m-2
                \end{pmatrix}.
\]
We let $M_n$ denote the standard parabolic associated to the partition $(m+1,m)$ of $n$ and set $H_n=w_n M_n w_n^{-1}$ as in the even case, we again write $h(g_1,g_2)=w_n diag(g_1,g_2)w_n^{-1}$ for $diag(g_1,g_2)$ in $M_n$. Note that the $H_n$ are compatible in the sense that $H_n\cap G_{n-1}=H_{n-1}$. \\
For $C\subset G_n$, we write $C^\sigma$ for $C\cap H_n$.
We will also need the matrix $w'_n$, which is the matrix of the permutation 
$$
\begin{pmatrix}  1 & 2 & \cdots & m  \! \ \ \ \ \ \ \ \ | &  m+1 & | & m+2 & m+3 & \cdots &  2m   \\ 
                 1 & 3 & \cdots & \!\! 2m-1 \ \ |&  2m  & | & 2   & 4   & \cdots &  2m-2 
                \end{pmatrix}
$$ when $n=2m$ is even, and 
of 
$$\begin{pmatrix}  1 & 2 & \cdots & m \! \ \ \ | &  m+1 & | & m+2 & m+3 & \cdots &  2m+1   \\ 
                 2 & 4 & \cdots &  2m  \ | &  2m +1 & | & 1   & 3   & \cdots &  2m-1 
                \end{pmatrix}$$ when $n=2m+1$ is odd.\\                  
In the sequel, $F$ will generally be a local field, whereas $\A$ will be the ring of adeles of a number field $k$. 
When $G$ is the points of an algebraic group defined over $\Z$ on $F$ or $\A$, we denote by $Sm(G)$ the category of smooth complex $G$-modules. 
Every representation we will consider from now on will be smooth and complex.\\
We will denote by $\delta_H$ the positive character of $N_G(H)$ such that if $\mu$ is a right Haar measure on $H$, and $int$ is the action given 
by $(int(n)f)(h)=f(n^{-1}hn)$, of $N_G(H)$ smooth functions $f$ with compact support on $H$, then $\mu \circ int(n)= \delta_H^{-1}(n)\mu $ for $n$ 
in $N_G(H)$.\\
If $G=G_{n}(\A)$, $H=H_n(\A)$, $\pi$ is a cuspidal representation of $G$ with trivial central character, and $\chi$ is a smooth character 
of $H$ trivial on $Z_n(A)$, we say that $\pi$ has an $(H,\chi)$-period if there is a cusp form in the space of $\pi$ such that the integral 
(which is convergent by Proposition 1 of \cite{AGR}) $$\int_{Z_{n}(\A)H_n(k)\backslash H_n(\A)} \phi(h)\chi^{-1}(h)dh$$ is nonzero.\\
For $m\in \N-\{0\}$, we will denote by $\mathcal{S}(F^m)$ the Schwartz space of functions (smooth and rapidly decreasing) on $F^m$ when $F$ is archimedean, and by 
$\sm_c(F^m)$ the Schwartz space of smooth functions with compact support on $F^m$ when $F$ is non-archimedean. We denote by 
$\mathcal{S}(\A^m)$ the space of Schwartz functions on $\A^m$, which is by definition the space of linear combinations of decomposable functions 
$\Phi=\prod_\nu \Phi_\nu$, with $\Phi_\nu$ in $\mathcal{S}(k_\nu^m)$ when $\nu$ is an archimedean place, and in $\sm_c(k_\nu^m)$ when 
$\nu$ is non-archimedean, with the extra condition that $\Phi_\nu=\1_{{\o_\nu}^m}$ for almost every non archimedean place $\nu$. On these spaces, there is a natural action of either $G_m(F)$, or $G_m(\A)$. In every case, if $\theta$ 
is a nontrivial character of $F$ or $\A/k$, we will denote by $\widehat{\Phi}^\theta$ or $\widehat{\Phi}$ the 
Fourier transform of a Schwartz function $\Phi$, with respect to a $\theta$-self-dual Haar measure.\\
If $\chi$ is a character of the local field $F$ of characteristic zero, with normalised absolute value $|.|_F$, we denote by 
$Re(\chi)$ the real number $r$ such that for all $x$ in $F^*$, one has $\sqrt{\chi(x)\overline{\chi(x)}}=|x|_F^r$. 
If $\chi$ is a character of $\A^*/k^*$, for $k$ a number field, and $\A$ its adele ring, we denote by $Re(\chi)$ the real number such that 
for all $x$ in $\A^*$, one has $\sqrt{\chi(x)\overline{\chi(x)}}=|x|^r$ for $|.|$ the norm of $\A^*$. \\

In the sequel, the equalities of two integrals, involving integration over quotients or subgroups,
 are valid up to correct normalisation of Haar measures.\\ 

\section{The local theory}\label{local}

We start with the non-archimedean case, which is studied in great detail in \cite{M4}. Here we just give the definition of the $L$-function, 
a non vanishing property needed for the global case, as well as the unramified computation.  

\subsection{The local non-archimedean $L$-function}

Let $\theta$ be a nonzero character of $F$. Let $\pi$ be a generic representation of $G_{n}$, $W$ belong to the Whittaker model $W(\pi,\theta)$,
and  $\Phi$ be a function in $C_c^{\infty}(F^{m})$. We denote by $\chi$ a character of $H_n$ of the form
 $h(h_1,h_2)\mapsto \alpha(det(h_1)/det(h_2))$, for $\alpha$ a character of $F^*$, and by $\d$ the character 
$h(h_1,h_2) \mapsto |h_1|/|h_2|$ of $H_n(F)$.\\
Denoting by $L_m$ the $m$-th row of a matrix, we define formally the integral 
$$\Psi(s,W,\chi,\Phi)= \int_{N_{n}\cap H_n \backslash H_n} W (h)\Phi(L_m(h_2))|h|^s\chi(h)\d(h)^{-1/2} dh.$$ 

This integral is convergent for $Re(s)$ large, and defines an element of $\C(q^{-s})$:

\begin{thrm}\label{nonarchL}
There is a real number $r_{\pi,\chi}$, such that each integral $\Psi(s,W,\chi,\Phi)$ converges for $Re(s)>r_{\pi,\chi}$. 
Moreover, when $W$ and $\Phi$ vary in $W(\pi,\theta)$ and $C_c^{\infty}(F^{m})$ respectively, they 
span a fractional ideal of $\C[q^{s},q^{-s}]$ in $\C(q^{-s})$, generated by an Euler factor which we denote 
$L^{lin}(s,\pi,\chi)$.\\
\end{thrm}

\begin{proof}
The convergence for $Re(s)$ greater than a real $r_{\pi,\chi}$ is classical, it is a consequence of the asymptotic expansion of the restriction 
of $W$ to the torus $A_n$, which can be found in \cite{JPS2} for example. 
The fact that these integrals span a fractinal ideal of $\C[q^{s},q^{-s}]$ is a consequence of the observation that $\Psi(s,W,\chi,\Phi)$ is 
multiplied by 
$|h|^{-s}\chi^{-1}(h)\d^{1/2}(h)$ when one replaces $W$ and $\Phi$ by their right translate under $h$.\\
Denoting by $c_\pi$ the central character of $\pi$, and by $K_n$ the points of 
$H_n$ on $\o$, we write, thanks to Iwasawa decomposition, the integral $\Psi(s,W,\chi,\Phi)$ as 
$$\int_{K_n}\int_{N_{n}^\sigma \backslash P_{n}^\sigma} 
W (pk)|p|^{s-1/2}\chi(pk)\left(\int_{F^*}\Phi(aL_m(k))c_{\pi}(a)|a|^{ns}da\right)dpdk$$
As in \cite{JPS2}, for any $\phi$ in $\mathcal{C}^\infty_c(N_{n} \backslash P_{n},\theta)$, there is a $W$ such that 
$W_{|P_{n}}$ is $\phi$. Such a $\phi$ is right invariant under an open subgroup $U$ of $K_n$, which also fixes $\chi$. We then choose 
$\Phi$ to be the characteristic function of $\{L_m(h_2), h_2\in U\}$, the integral then reduces to a positive multiple of 
$$\int_{N_{n}^\sigma \backslash P_{n}^\sigma} \phi (p)|p|^{s-1/2}\chi(p)dp.$$ 
We now see that for $\phi$ well-chosen, this last integral is $1$, i.e. $\Psi(s,W,\chi,\Phi)$ is $1$. This implies that the generator of 
the fractional ideal spanned by the $\Psi(s,W,\chi,\Phi)$ can be chosen as an Euler factor.\\
\end{proof}
\begin{remark}\label{differentnotations}
There is a notational difference with \cite{M4}. In [loc. cit.], we call $L^{lin}(s,\pi,\chi)$ what we call $L^{lin}(s,\pi,\chi\delta^{1/2})$ here. 
\end{remark}

We have the following corollary to the previous proof.

\begin{cor}\label{constant}
There are $W\in W(\pi,\theta)$, and $\phi$ the characteristic function of a neighbourhood of 
$(0,\dots,0,1)\in F^n$, such that $\Psi(s,W,\chi,\Phi)$ is equal to $1$ in $\C(q^{-s})$. 
\end{cor}

\subsection{The unramified computation}\label{unramified}

Here we show that the local Rankin-Selberg integrals give the expected $L$-function at the unramified places.\\
Let $\pi^0$ be an unramified generic representation of $GL(n,F)$, and $W^0$ the normalised 
spherical Whittaker function in $W(\pi^0,\theta)$ (here $\theta$ has conductor $\mathfrak{O}$), and let $\Phi^0$ the characteristic function 
of $\mathfrak{O}^n$. We will use the notations of Section 3 of \cite{F}. We recall that $\pi^0$ is a commuting product 
(in the sense of \cite{BZ2}, i.e. corresponding to normalised parabolic induction)
$\chi_1 \times\dots \times \chi_{n}$ of unramified characters, and we denote $\chi_i(\w)$ by $z_i$. 
Then it is well-known (see \cite{Sat}) that if $\l$ is an element of $\Z^{n}$, then $W(\w^{\l})$ is zero unless $\l$ belongs 
to the set $\Lambda^+$ consisting of $\l$'s satisfying $\l_1\geq\dots \geq \l_{n}$, in which case 
$W(\w^{\l})=\d_{B_{n}}^{1/2}(\w^{\l})s_{\l}(z)$, where $s_\l(z)=det(z_i^{\l_j+n-j})/det(z_i^{n-j})$.\\
In this case, using Iwasawa decomposition, denoting by $\Lambda^{++}$ the subset of $\Lambda^{+}$ with $\lambda_{n}\geq 0$, 
and writing $a'$ for $(a_1,a_3,\dots,a_{n-1})$ and $a''$ for $(a_2,a_4,\dots,a_{n})$, 
one has  the identities 
$$\Psi(s,W^0,\chi,\Phi^0)=\int_{A_{n}} W^0(a)\d_{B_n}^{-1}(a')\d_{B_n}^{-1}(a'')\chi(a)\Phi^0(a_{n})\chi_0(a)|a|^s\d(a)^{-1/2}da$$
$$=\int_{A_{n}} W^0(a)\chi_0(a)\Phi^0(a_{n})\d_{B_{n}}^{-1/2}(a)\alpha(det(a'))\alpha^{-1}(det(a''))|a|^sda$$
$$=\sum_{\l\in \Lambda^{++}}s_{\l}(z)q^{-s.tr\l}\alpha(\w)^{\sum_{i=1}^m (\l_{2i-1}-\l_{2i})}=
\sum_{\l\in \Lambda^{++}}s_{\l}(q^{-s}z)\alpha(\w)^{c(\l)},$$ 
where $c(\l)=\sum_{i=1}^m (\l_{2i-1}-\l_{2i})$.
We now refer to Example 7 of p.78 in \cite{Mac}, which asserts that if $x=(x_1,\dots,x_n)$ is a vector with nonzero complex coordinates, 
and if $t$ is a complex number, then the sum $\sum_{\l\in \Lambda^{++}}s_{\l}(x)t^{c(\l)}$ is equal to 
$$\prod_i (1-tx_i)\prod_{j<k} (1-x_jx_k).$$ In particular, with $x_i=z_iq^{-s}$, and $t=\alpha(\w)$, we obtain that $\Psi(s,W^0,\chi,\Phi^0)$ 
is equal to 
$$\prod_i (1-\alpha(\w)z_iq^{-s})\prod_{j<k} (1-z_j z_kq^{-2s})=L(\alpha\otimes \pi^0,s)L(\pi^0,\Lambda^2,2s).$$

We end with the archimedean theory.

\subsection{Convergence and non vanishing of the archimedean integrals}\label{archi}

Here $F$ is archimedean, $\theta$ is a unitary character of $F$, and $\pi$ is a generic unitary representation of $G_{n}$, as in Section 2 of \cite{JS1}, to which we refer concerning this vocabulary. We denote by $W(\pi,\theta)$ its smooth Whittaker model.\\
We denote by $\d$ the character $h(h_1,h_2) \mapsto |h_1|/|h_2|$ of $H_n(F)$, and by $\chi$ the character 
 $h(h_1,h_2) \mapsto \alpha(det(h_1))/\alpha((det(h_2))$, for $\alpha$ a character of $F^*$.\\

We now define formally the following integral, for $W$ in $W(\pi,\theta)$, and $\Phi$ in $\mathcal{S}(F^n)$:
$$\Psi(s,W,\chi,\Phi)= \int_{N_{n}\cap H_n \backslash H_n} W (h)\Phi(L_m(h_2))\chi(h)|h|^s \d(h)^{-1/2} dh.$$

We first state a proposition concerning the convergence of this integral:     

\begin{prop}\label{convarchi} 
For $Re(\alpha)\geq 0$, there is a positive real $\e$ independant of $W$ and $\Phi$, such that the integral $\Psi(s,W,\chi,\Phi)$ is absolutely convergent for $s\geq 1/2-\e$. 
In particular it defines a holomorphic function on this half plane. 
\end{prop}
\begin{proof}
It is a consequence of Iwasawa decomposition that to prove this statement it is enough to prove it for the integral 
$$\int_{A_{n-1}} W (a)|a|^s\d_{B_{n}}^{-1}(a')\d_{B_{n}}^{-1}(a'')\alpha(det(a'))\alpha^{-1}(det(a''))\d(a)^{-1/2} da,$$ 
where $a'=(a_1,\dots,a_{n-1})$, and $a''=(a_2,\dots,a_{n-2},1)$. However we have the equality 
$\d_{B_{n}}^{-1}(a')\d_{B_{n}}^{-1}(a'')\d(a)^{-1/2}=\d_{B_{n}}^{-1/2}(a)=\d_{B_{n-1}}^{-1/2}(a)|a|^{-1/2}$. But according to Section 4 of 
\cite{JS2}, writing $t_i$ for $a_i/a_{i+1}$ there is a finite set $X$ consisting of functions which are products of 
polynomials in the logarithm of the $|t_i|$'s and a character $\chi(a)=\prod_{i=1}^{n-1}\chi_i(t_i)$ with $Re(\chi_i)>0$, such that $|W(a)|$ is 
majorised by a sum of functions of the form $S(t_1,\dots,t_{n-1})\d_{B_{n-1}}^{1/2}(t)C_\chi(t)$, where $S$ 
is a Schwartz function on $F^{n-1}$, and $C_{\chi}$ belongs to $X$. Hence we only need to consider the convergence of 
$$\int_{A_{n-1}} \!\!\!\!\!\! C_\chi(t(a)) S(t(a))\alpha(det(a')/det(a''))|a|^{s-1/2} da$$ 
$$= \int_{A_{n-1}} \!\!\!\!\!\! C_\chi(t) S(t)\prod_{j=1}^n\alpha(t_{2j-1})\prod_{i=1}^{n-1}|t_i|^{i(s-1/2)} d t.$$
The statement follows, taking $\e=min(Re(\chi_j))$ for $C_\chi$ in $X$. 
\end{proof}

Now we state our second result, about the nonvanishing of our integrals at $1/2$ for good choices of $W$ and $\Phi$. The proof of this proposition, as that of the previous one, will be an easy adaptation of the techniques of \cite{JS1}, though we followed even more closely the version of \cite{FZ}. 

\begin{prop}\label{nonvanish}
Suppose that $Re(\alpha)\geq 0$, and let $s$ be a complex number with $Re(s)\geq 1/2-\e$. There is $W$ in $W(\pi,\theta)$, and $\Phi$ in $\mathcal{S}(F^n)$, such that $\Psi(s,W,\chi,\Phi)$ is nonzero. Moreover one can choose $\Phi\geq 0$, so that 
$\widehat{\Phi}(0)>0$.
\end{prop}
\begin{proof}
If not, $\Psi(s,W,\chi,\Phi)$ is zero for every $W$ in $W(\pi,\theta)$, and $\Phi\geq 0$ in $S(F^n)$. We are first going to prove that this implies that 
$$\int_{N_{n-1}^\sigma\backslash H_{n-1}} W \begin{pmatrix} h& \\ &1 \end{pmatrix}\chi(h) |h|^{s-1/2} dh=0$$ for every $W$ in $W(\pi,\theta)$. 
Indeed, one has 
$$\int_{N_{n}^\sigma\backslash G_{n}^\sigma} \!\!\!\!\!\!\!\!\!\!\!\!
 W(h)\Phi(L_m(h_2))\chi(h)|h|^s \d(h)^{-1/2}dh=\!\! \int_{N_n^\sigma\backslash G_{n}^\sigma}\!\!\!\!\!\! \!\!\!\!\!\!W(h)\Phi(L_m(h))\chi(h)|h|^{s-1/2}|h_2|dh$$ 
$$=\int_{P_{n}^\sigma\backslash  G_{n}^\sigma} \left(\int_{N_{n}^\sigma\backslash P_{n}^\sigma} W(ph)\chi(ph)|ph|^{s-1/2}dp\right)
\Phi(L_m(h_2))d\bar{h}$$ 
where $|h_2|dh$ is quasi-invariant on $N_{n}^\sigma\backslash G_{n}^\sigma$, and $d\bar{h}$ is quasi-invariant on 
$P_{n}^\sigma\backslash  G_{n}^\sigma$.\\

But $P_{n}^\sigma\backslash  G_{n}^\sigma \simeq F^n-\{0\}$ via $\bar{h}\mapsto L_n(h_2)$, and the Lebesgue measure on 
$F^n-\{0\}$ corresponds to $d\bar{h}$ via this homeomorphism. Hence, denoting 
$G(\bar{h})=  \int_{N_{n}^\sigma\backslash  P_{n}^\sigma} W(ph)\chi(ph)|ph|^{s-1/2}dp$, one has that for every $\Phi$: 
$$\int_{F^n-\{0\}} G(x)\Phi(x)dx=0,$$ in particular $G(0,\dots,0,1)=0$ 
(taking $\Phi\geq 0$ approximating the Dirac measure supported at $(0,\dots,0,1)$), hence
 $\int_{N_{n}^\sigma\backslash  P_{n}^\sigma} W(p)\chi(p)|p|^{s-1/2}dp=0$.\\
Then, one checks 
(see Section 2 of \cite{JS1}), that for every $\Phi\in \mathcal{S}(F^{n-1})$, the map 
$W_{\phi}: g\mapsto \int_{F^{n-1}}W(gu^\sigma(x))\Phi(x)dx$, where $u^\sigma$ is the natural isomorphism between $F^{n-1}$ and $U_{n}^\sigma$, 
belongs to $W(\pi,\theta)$ again. But 
$W_{\phi}\begin{pmatrix} h& \\ &1 \end{pmatrix} = W \begin{pmatrix} h& \\ &1 \end{pmatrix}\widehat{\Phi}(L_m(h_1))$ for $h$ in $H_{n-1}$, hence 
$$\int_{N_{n-1}^\sigma\backslash H_{n-1}} W \begin{pmatrix} h& \\ &1 \end{pmatrix}\widehat{\Phi}(L_m(h_1))\chi(h) |h|^{s-1/2} dh$$ 
is zero for every $W$ and $\Phi$, which in turn implies the equality 
$$\int_{N_{n-2}^\sigma\backslash H_{n-2}} W \begin{pmatrix} h& \\ &I_2 \end{pmatrix}\chi(h)|h|^{s-1/2} dh=0$$ for every $W$ in $W(\pi,\theta)$. Continuing the process, we obtain 
$W(I_{n})=0$ for every $W$ in $W(\pi,\theta)$, a contradiction. We didn't check the convergence of our integrals at each step, but it follows from Fubini's theorem.
\end{proof}

\section{The global theory}\label{global}

\subsection{The Eisenstein series}
In the global case, let $\pi$ be a smooth automorphic cuspidal representation of $G(\A)$ with trivial central character, 
$\phi$ a cusp form in the space of $\pi$, and $\Phi$ an element of the Schwartz space $\mathcal{S}(\A^{n})$. We denote by
 $\chi$ a character of $H_n(\A)$ of the form $h(h_1,h_2)\mapsto \alpha(det(h_1)/det(h_2))$ for $\alpha$ a character of $\A^*/k^*$, 
and by $\d$ again the character 
$h(h_1,h_2) \mapsto |h_1|/|h_2|$ of $H_n(\A)$. Then we define 
$$f_{\chi,\Phi}(s,h)=|h|^s\chi(h)\d^{-1/2}(h)\int_{\A^*}\Phi(aL_m(h_2))|a|^{ns}d^*a$$ for $h$ in $H$, $s$ in 
the half plane $Re(s)>1/{n}$, where the integral converges absolutely. It is obvious that $f_{\chi,\Phi}(s,h)$ 
is $Z_{n}(k)P_{n}^\sigma(k)$-invariant on the left.\\
Now we average $f$ on $H_n(k)$ to obtain the following Eisenstein series: 
$$E(s,h,\chi,\Phi)=\sum_{\gamma\in Z_{n}(k)P_{n}^\sigma(k)\backslash H_n(k)} f_{\chi,\Phi}(s,\gamma h).$$ 
One can rewrite $E(s,h,\chi,\Phi)$ as  
$$|h|^s\chi(h)\d^{-1/2}(h)\int_{k^*\backslash \A^*}\Theta'_\Phi(a,h)|a|^{ns}d^*a,$$ 
where $\Theta'_\Phi(a,h)=\sum_{\xi\in k^n-\{0\}}\Phi(a \xi h_2)$.\\
According to Lemmas 11.5 and 11.6 of \cite{GJ}, it is absolutely convergent for $Re(s)>1/2$,
 uniformly on compact subsets of $H_n(k)\backslash H_n(\A)$, and of moderate growth with respect to $h$.\\

Write $\Theta_\Phi(a,h)$ for $\Theta'_\Phi(a,h)+\Phi(0)$, then the Poisson formula for $\Theta_\Phi$ gives:

$$\Theta_\Phi(a,h)=|a|^{-n}|h_2|^{-1}\Theta_{\hat{\Phi}}(a^{-1},^t\!h^{-1}).$$

This allows us to write, for $c$ a certain nonzero constant: 

$$E(s,h,\chi,\Phi)=|h|^s\chi(h)\d^{-1/2}(h)\int_{|a|\geq 1}\Theta'_\Phi(a,h)|a|^{ns}d^*a+$$
$$|h|^{s-1/2}\chi(h)\int_{|a|\geq 1}\Theta'_{\widehat{\Phi}}(a,^t\!h^{-1})|a|^{n(1-2s)}d^*a+u(s)$$

with $u(s)=-c\Phi(0)|h|^s\chi(h)\d^{-1/2}(h)/2s+c\widehat{\Phi}(0)\chi(h)|h|^{s-1/2}/(1-2s)$.\\

We deduce from this, appealing again to Lemma 11.5 of \cite{GJ}, the following proposition:

\begin{prop} 
$E(s,h,\chi,\Phi)$ admits a meromorphic extension to $\C$, has at most simple poles at $0$ and $1/2$, and satisfies the functional equation:

$$E(1/2-s,^t\!h^{-1},\chi^{-1}\d^{1/2},\widehat{\Phi})=E(s,h,\chi,\Phi).$$
\end{prop}

Then the following integral converges absolutely for $Re(s)>1/2$:  $$I(s,\phi,\chi,\Phi)=\int_{Z_{n}(\A)H_n(k)\backslash H_n(\A)} E(s,h,\chi,\Phi)\phi(h)dh.$$

\begin{thrm}\label{I}
The integral $I(s,\phi,\chi,\Phi)$ extends meromorphically to $\C$, with poles at most simple at $0$ and $1/2$, moreover, a pole at $1/2$ occurs if and only if 
the global $\chi^{-1}$-period $$\int_{Z_{n}(\A)H_n(k)\backslash H_n(\A)}\chi(h) \phi(h)dh$$ is not zero, and $\widehat{\Phi}(0)\neq 0$. 
The integral $I(s,\phi,\chi,\Phi)$ also admits the following functional equation:
$$I(1/2-s,\tilde{\phi},\chi^{-1}\d^{1/2},\widehat{\Phi})=I(s,\phi,\chi,\Phi),$$ 
where $\tilde{\phi}:g\mapsto \phi(^t\!g^{-1})$.
\end{thrm} 
\begin{proof}
It is clear that the residue of $I(s,\phi,\chi,\Phi)$ at $1/2$ is 
$$c\widehat{\Phi}(0)\int_{Z_{n}(\A)H_n(k)\backslash H_n(\A)}\chi(h) \phi(h)dh,$$ hence the result about periods. 

\end{proof}

\subsection{The Euler factorisation}\label{unfolding}
 
Let $\theta$ be a nontrivial character of $\A/k$, we denote by $W_{\phi}$ the Whittaker function on $G_{n}(\A)$ 
associated to $\phi$, and we let
 $$\Psi(s,W_{\phi},\chi,\Phi)=\int_{N_{n}^\sigma (\A)\backslash H_n(\A)} W(h)\Phi(\eta h)\chi(h)|h|^s\d^{-1/2}(h) dh,$$ 
this integral converges absolutely for $Re(s)$ large by classical gauge estimates of Section 13 of \cite{JPS}, and is the product of 
the similar local integrals. We will need the following expansion of cusp forms on the mirabolic subgroup, which can be found in \cite{C}, p.5.

\begin{prop}\label{Fourier}
 Let $\phi$ be a cusp form on $P_l(\A)$, then one has the following partial Fourier expansion with uniform convergence 
for $p$ in compact subsets of $P_l(\A)$:
$$\phi(p)=\sum_{\g\in P_{l-1}(k)\backslash G_{l-1}(k)}
\left(\int_{y\in (\A/k)^{l-1}} \phi(u(y)\begin{pmatrix} \g & \\ & 1 \end{pmatrix} p)\theta^{-1}(y)dy\right)$$
\end{prop}

We now unravel the integral $I(s,\phi,\chi,\Phi)$, following the strategy in 
\cite{BF} and \cite{JS2}, which is to unravel step by step.

\begin{thrm}\label{egal}
 One has the identity $I(s,\phi,\chi,\Phi)= \Psi(s,W_{\phi},\chi,\Phi)$ for $Re(s)$ large.
\end{thrm}
\begin{proof}
We suppose that $s$ is large enough so that $I(s,\phi,\chi,\Phi)$ is absolutely convergent. 
Denoting by $\chi_s$ the character $\chi \d^{-1/2}|.|^s$ of $H_n(\A)$, we start with 
$$I(s,\phi,\chi,\Phi)=\!\!\!\!\!\!\!\!\!\! \int\limits_{Z_{n}(\A)H_n(k)\backslash H_n(\A)}\!\!\!\!\!\!\!\!\!\! E(s,h,\chi,\Phi)\phi(h)dh
= \!\!\!\!\!\!\!\!\!\! \int\limits_{Z_{n}(\A)P_{n}^\sigma(k)\backslash H_n(\A)} \!\!\!\!\!\!\!\!\!\! f_{\chi,\Phi}(s,h)\phi(h)dh $$
$$= \!\!\!\!\!\!\!\!\!\!  \int\limits_{P_{n}^\sigma(k) \backslash H_n(\A)} \!\!\!\!\!\!\!\!\!\! \Phi(Ln(h_2))\phi(h)\chi_s(h)dh $$
We denote for the moment $\Phi(L_m(h_2))\chi_s(h)$ by $F(h)$, and for $l$ between $0$ and $m-1$, we write:
$$I_l=\!\!\!\!\!\!\!\!\!\! \int\limits_{P_{2l}^\sigma(k) (U_{2l+1}\dots U_{n})^\sigma(\A) \backslash H_n(\A)} \!\!\!\!\!\!\!\!\!\!\!\!F(h)
\left(\int\limits_{(U_{2l+1}\dots U_{n})(k)\backslash (U_{2l+1}\dots U_{n})(\A)}  \!\!\!\!\!\!\!\phi(nh)\theta^{-1}(n)dn\right)dh,$$
in particular $I_0=\Psi(W_{\phi},\Phi,s)$. We also write $I_m=I(s,\phi,\chi,\Phi)$.\\

To prove the theorem, we only need to prove that $I_{l}=I_{l+1}$ for $0\leq m-1$, that's what we do now, we will see that the absolute convergence 
of $I_l$ (i.e. the fact that $F(h)
\left(\int\limits_{(U_{2l+1}\dots U_{n})(k)\backslash (U_{2l+1}\dots U_{n})(\A)}  \!\!\!\!\!\!\!\phi(nh)\theta^{-1}(n)dn\right)$ is absolutely 
integrable over the quotient $P_{2l}^\sigma(k) (U_{2l+1}\dots U_{n})^\sigma(\A) \backslash H_n(\A)$) implies that of $I_{l+1}$ during the process. We will tacitly use several times the fact that if $G$ is a unimodular locally compact group, $K<H$ closed unimodular 
subgroups of $G$, and $A$ is a continuous integrable function on $K\backslash G$, then $B(x)=\int_{K\backslash H} A(hx) dh$ is absolutely convergent for all $x\in G$, integrable over $H\backslash G$, and one has 
$\int_{H\backslash G} B(h)dh=\int_{K\backslash G} A(g)dg$.\\
 We thus suppose that $I_l$ is absolutely convergent, so one can write $I_l$ 
as the absolutely convergent sum  
$$I_l=\sum_{\g}\int\limits_{P_{2l+1}^\sigma(k) (U_{2l+1}\dots U_{n})^\sigma(\A) \backslash H_n(\A)} \!\!\!\!\!\!\!\!\!\!\!\! \!\!\!\!\!\!\! F(h)
\left(\int\limits_{(U_{2l+1}\dots U_{n})(k)\backslash (U_{2l+1}\dots U_{n})(\A)} \!\!\!\!\!\!\!\!\!\!\!\!\!\!\!\!\!\!\!\!\!\!\!\!\!\!\!\!\!\!\!\phi(n\g h)\theta^{-1}(n)dn\right)dh$$
$$= \int\limits_{P_{2l+1}^\sigma(k) (U_{2l+1}\dots U_{n})^\sigma(\A) \backslash H_n(\A)} \!\!\!\!\!\!\!\!\!\!\!\! \!\!\!\!\!\!\! F(h)
\left(\sum_{\g} \int\limits_{(U_{2l+1}\dots U_{n})(k)\backslash (U_{2l+1}\dots U_{n})(\A)} \!\!\!\!\!\!\!\!\!\!\!\!\!\!\!\!\!\!\!\!\!\!\!\!\!\!\!\!\!\!\!\phi(n\g h)\theta^{-1}(n)dn\right)dh,$$ where the sum is over $\g\in P_{2l}^\sigma U_{2l+1}^\sigma(k)\backslash P_{2l+1}^\sigma(k)$. 
Now, a system of representatives of $P_{2l}^\sigma(k)U_{2l+1}^\sigma(k) \backslash P_{2l+1}^\sigma(k)$ is given by the elements $w'_{2l+1}(\g)$, for 
$\g$ in $P_{l}(k)\backslash G_{l}(k)$.\\
We now apply Proposition \ref{Fourier} at $p=1$ to the cusp form:
$$p\in P_{l+1}\mapsto \!\!\!\!\!\!\!\!\!\!\!\!\!\!\!\!\!\!\!\! \int\limits_{(U_{2l+2}\dots U_{n})(k)\backslash (U_{2l+2}\dots U_{n})(\A)} 
\!\!\!\!\!\!\!\!\!\!\!\!\!\!\!\!\!\!\!\! \phi(nw'_{2l+1}(p)h)\theta^{-1}(n)dn.$$  
We thus obtain the relation, the sum being absolutely convergent by Proposition \ref{Fourier}:
$$\sum_{\g\in P_l(k)\backslash G_l(k)} \left(\int\limits_{(U_{2l+1}U_{2l+2}\dots U_{n})(k)\backslash (U_{2l+1}U_{2l+2}\dots U_{n})(\A)} 
 \!\!\!\!\!\!\!\phi(nw'_{2l+1}(\g) h)\theta^{-1}(n)dn\right)$$ 
$$=\int\limits_{u\in U_{2l+1}^\sigma(k)\backslash U_{2l+1}^\sigma(\A)}\left(\int\limits_{n\in(U_{2l+2}\dots U_{n})(k)\backslash (U_{2l+2}\dots U_{n})(\A)}
 \!\!\!\!\!\!\!\phi(nuh)\theta^{-1}(n)dn\right) du.$$

Replacing in $I_l$, one obtains the equality, with $J_l$ absolutely convergent:
$$I_l=J_l=\!\!\!\!\!\!\!\!\!\!\!\!\!\!\!\!\!\!\!\! \!\!\!\!\!\!\!\!\!\!\int\limits_{P_{2l+1}^\sigma(k) (U_{2l+2}\dots U_{n})^\sigma(\A) \backslash H_n(\A)} \!\!\!\!\!\!\!\!\!\!\!\!F(h)
\left(\int\limits_{(U_{2l+2}\dots U_{n})(k)\backslash (U_{2l+2}\dots U_{n})(\A)}  \!\!\!\!\!\!\!\phi(nh)\theta^{-1}(n)dn\right)dh.$$

Again, we have 
$$J_l=\sum_{\g}\int\limits_{P_{2l+2}^\sigma(k) (U_{2l+2}\dots U_{n})^\sigma(\A) \backslash H_n(\A)} \!\!\!\!\!\!\!\!\!\!\!\! \!\!\!\!\!\!\! F(h)
\left(\int\limits_{(U_{2l+2}\dots U_{n})(k)\backslash (U_{2l+2}\dots U_{n})(\A)} \!\!\!\!\!\!\!\!\!\!\!\!\!\!\!\!\!\!\!\!\!\!\!\!\!\!\!\!\!\!\!\phi(n\g h)\theta^{-1}(n)dn\right)dh,$$ where the sum is now over $\g\in P_{2l+1}^\sigma(k) U_{2l+2}^\sigma(k)\backslash P_{2l+2}^\sigma(k)$, a system of 
representatives of which is given by the elements $w'_{2l+2}(\g)$, for $\g$ in $P_{l+1}(k)\backslash G_{l+1}(k)$. Applying again Proposition \ref{Fourier} at $p=1$, this time to the cusp form: 
$$p\in P_{l+2}\mapsto
\int\limits_{(U_{2l+3}\dots U_{n})(k)\backslash (U_{2l+3}\dots U_{n})(\A)} \phi(nw'_{2l+2}(p)h)\theta^{-1}(n)dn$$ one has 
$$\sum_{\g\in P_{l+1}(k)\backslash G_{l+1}(k)} \left(\int\limits_{(U_{2l+2}U_{2l+3}\dots U_{n})(k)\backslash (U_{2l+2}U_{2l+3}\dots U_{n})(\A)}  \!\!\!\!\!\!\!\phi(nw'_{2l+2}(\g) h)\theta^{-1}(n)dn\right)$$ 
$$= \int\limits_{u\in U_{2l+2}^\sigma(k)\backslash U_{2l+2}^\sigma(\A)} 
\left( \int\limits_{n\in(U_{2l+3}\dots U_{n})(k)\backslash (U_{2l+3}\dots U_{n})(\A)}  \!\!\!\!\!\!\!\phi(nuh)\theta^{-1}(n)dn\right)du.$$
Finally replacing in $J_l$, one obtains $J_l=I_{l+1}$, together with the absolute convergence of $I_{l+1}$. Notice that when 
$l=m-1$, the last bit of the proof becomes 
$$J_m=\int\limits_{P_{n-1}^\sigma(k) U_{n}^\sigma(\A) \backslash H_n(\A)} F(h)
\left(\int\limits_{U_n(k)\backslash U_{n}(\A)}  \phi(nh)\theta^{-1}(n)dn\right)dh$$
$$\!\!\!\!\!\!\!=\int\limits_{P_{n}^\sigma(k)  U_{n}^\sigma(\A) \backslash H_n(\A)} F(h)\sum_{\g\in w'_{n}(P_{n-1}(k)\backslash G_{n-1}(k))}
\left(\int\limits_{ U_{n}(k)\backslash U_{n}(\A)} \phi(n\g h)\theta^{-1}(n)dn\right)dh$$
$$\!\!\!\!\!\!\!=\!\!\!\!\!\!\!\!\!\!\!\!\!\!\int\limits_{P_{n}^\sigma(k)  U_{n}^\sigma(\A) \backslash H_n(\A)}\!\!\!\!\!\!\! F(h)
\left(\int\limits_{ U_{n}^\sigma(k)\backslash U_{n}^\sigma(\A)} \phi(nh)\theta^{-1}(n)dn\right)dh =\!\!\!\!\!\!\!\!\!\!\!\!\!\! \int\limits_{P_{n}^\sigma(k)(\A) \backslash H_n(\A)} F(h)
\phi(h)dh  =I_m.$$
\end{proof}

As a corollary, we see that $\Psi(s,W_{\phi},\chi,\Phi)$ extends to a meromorphic function (namely $I(s,\phi,\chi,\Phi)$). Writing 
$\pi$ as the restricted tensor product $\otimes_{\nu} \pi_\nu$, for any $W=\prod_\nu W_\nu$ in $W(\pi,\theta)$, any decomposable 
$\Phi=\prod_\nu \Phi_\nu$ in $\mathcal{S}(\A^{n})$, one has $$\Psi(s, W,\chi,\Phi)= \prod_\nu \Psi(s,W_\nu,\chi_\nu,\Phi_\nu).$$

\subsection{The partial $L$-function}

Let $\pi=\otimes_{\nu} \pi_\nu$ be a cuspidal automorphic representation of $G_{n}(\A)$, and $S$ the finite set of places of $k$, 
such that for $\nu$ in $S$, $\pi_\nu$ is archimedean or ramified. Let $\alpha=\otimes_\nu \alpha_\nu$ be a character of $\A^*/k^*$ 
(with $\alpha_\nu$ unramified for almost all $\nu$), and 
$\chi$ be the character $h(h_1,h_2)\mapsto \alpha(det(h_1)/det(h_2))$ of $H_n(\A)$. We define the partial $L$-function 
$L^{lin,S}(s,\pi,\chi)$ to be the 
product $$L^{lin,S}(s,\pi,\chi)=\prod_{\nu \notin S} L(s,\alpha_\nu \otimes \pi_{\nu})L(2s,\Lambda^2,\pi_{\nu}),$$ where 
$L(s,\alpha_\nu \otimes \pi_{\nu})$ and $L(2s,\Lambda^2,\pi_{\nu})$ are 
the corresponding $L$-functions of the Galois parameter of $\pi_{\nu}$. Hence, if $\theta_\nu$ has conductor $\o_\nu$ at every unramified 
place $\nu$, the function $L^{lin,S}(s,\pi,\chi)$ is the product 
$\prod_{\nu \notin S} \Psi(s,W_{\nu}^0,\chi_\nu,\Phi_\nu^0)$ thanks to the unramified computation. Because of this, we see that it is meromorphic (it is equal 
to $\Psi(s,W,\Phi)/\prod_{\nu \in S} \Psi(s,W_{\nu},\Phi_\nu)$ for a well chosen $\Phi$ an $W$). We now show that when $Re(\alpha)\geq 0$, 
the partial $L$ function $L^{lin,S}(s,\pi,\chi)$ has at
 most a simple pole at $1/2$, and that this happens 
if and only if $\pi$ admits a $\chi^{-1}$-period.

\begin{thrm}\label{partialpole}
Suppose that $Re(\alpha)\geq 0$, the partial $L$-function $L^{lin,S}(s,\pi,\chi)$ is holomorphic
 for $Re(s)>1/2$ and has a pole at $1/2$ if and 
only if $\pi$ has an $(H_n(\A),\chi^{-1})$-period. If it is the case, this pole is simple.
\end{thrm}
\begin{proof}
Let $S_{\infty}$ be the archimedean places of $k$, and $S_f$ the set of finite places in $S$. First, for $\nu$ in $S_f$, thanks to Corollary 
\ref{constant}, we take $W_\nu$ and $\Phi_\nu$ such that $\Psi(s,W_{\nu},\chi_\nu,\Phi_\nu)=1$ for all $s$ 
(and $\widehat{\Phi_\nu}(0)> 0$ because $\Phi_\nu$ is positive). For 
any $s_0\geq 1/2$, if $\nu$ belongs to $S_\infty$, it is possible to take $W_{\nu,s_0}$ and $\Phi_{\nu,s_0}$ with $\widehat{\Phi_{\nu,s_0}}(0)>0$ such 
that $\Psi(s,W_{\nu,s_0},\chi_\nu,\Phi_{\nu,s_0})$ is convergent for $s\geq 1/2-\epsilon$, and $\neq 0$ for $s=s_0$ according to Propositions
\ref{convarchi} and \ref{nonvanish}. We write 
$$W_{s_0}=\prod_{\nu\in S_\infty} W_{\nu,s_0}\prod_{\nu\in S_f} W_{\nu}\prod_{\nu\notin S} W_{\nu}^0,$$ 
notice that $W_{s_0}$ is equal to $W_\phi$ for some cusp $\phi$ in the space of $\pi$, thanks to the restricted tensor product decomposition of $\pi$, 
which by multiplicity $1$ for local Whittaker functionals, implies the same decomposition of the global Whittaker model $W(\pi,\theta)$. If we now write  
$$\Phi_{s_0}=\prod_{\nu\in S_\infty} \Phi_{\nu,s_0}\prod_{\nu\in S_f} \Phi_{\nu}\prod_{\nu\notin S} \Phi_{\nu}^0,$$ the theorem follows 
from the equality 
$\prod_{\nu\in S}\Psi(s,W_{\nu,s_0},\chi_\nu,\Phi_{\nu,s_0}) L^{lin,S}(s,\pi)= \Psi(s,W_{s_0},\chi,\Phi_{s_0})$, and Theorems \ref{egal} and \ref{I}, as 
$\widehat{\Phi}(0)\neq 0$ (it is positive).
\end{proof}

We then recall the following Lemma.

\begin{lem}\label{selfdual}
 If the partial exterior square $L$-function $L^S(\pi,\Lambda^2,\pi)$ has a pole at $1$, then $\pi$ is self-dual.
\end{lem}
\begin{proof}
 It is a consequence of the main theorem of \cite{JS2} that if $L^S(\pi,\Lambda^2,\pi)$ has a pole at $1$, then 
$\pi$ has a non-vanishing Shalika period, hence all its non-archimedean 
components admit a Shalika model. Then, by \cite{JR}, all the non archimedean components must be self-dual, 
hence $\pi$ as well by strong multiplicity one.
\end{proof}

We will use it in the following proof, to remove the assumption $Re(\alpha)\geq 0$. We reobtain a theorem of 
Friedberg and Jacquet (\cite{FJ}), using directly the Bump-Friedberg $L$-function. 

\begin{thrm}\label{FJ}
The cuspidal automorphic representation $\pi$ of $G_{n}(\A)$ admits a global $\chi^{-1}$-period if and only if 
$L^S(s,\Lambda^2,\pi)$ has a pole at $1$, and $L(1/2,\alpha \otimes \pi)\neq 0$.
\end{thrm}
\begin{proof}
We first give the proof for $\alpha$ satisfying $Re(\alpha)\geq 0$, we will deduce the general case form this particular one. It is well know (see \cite{GJ}) that $L(s,\alpha\otimes \pi)$ is entire, hence $L^S(s,\alpha\otimes \pi)$. 
Moreover $L(1/2,\alpha\otimes \pi)=0$ if and only if $L^S(1/2,\alpha\otimes \pi)=0$. 
Indeed $L^S(s,\alpha\otimes \pi)$ is an entire multiple of $L(s,\alpha\otimes \pi)$, hence one implication. 
Using the Rankin-Selberg convolution for $G_{n}(\A)\times \A^*$, 
then for any $W$ in $W(\pi,\theta)$, denoting 
$\int_{\A^*}W(a,1,\dots,1)\alpha(a)|a|^{(n-1)/2}d^*a$ by 
$\Psi(s,W,\alpha)$, $\prod_{\nu\in S} \int_{k_\nu^*}W(t_\nu,1,\dots,1)\alpha_\nu(t_\nu)|t_\nu|^{(n-1)/2}d^*t_\nu$ by $\Psi(s,W_S,\alpha_S)$, and 
$\prod_{\nu\in S} L(s,\alpha_\nu\otimes \pi_\nu)$ by $L_S(s,\alpha\otimes \pi)$, one has $\Psi(s,W_S,\alpha_S)L^S(s,\alpha\otimes \pi)=
[\Psi(s,W_S,\alpha_S)/L_S(s,\alpha \otimes \pi)]L(s,\alpha \otimes \pi)$. But 
there is $\epsilon>0$ such that $\Psi(s,W_{S},\alpha_S)$ converges for $Re(s)>1/2-\e$ according to the estimates for the $W_\nu$'s 
retriction to $A_{n-1}$ given in Proposition 3 of \cite{JS2}. Hence if $L^S(1/2,\alpha\otimes \pi)=0$, then 
$[\Psi(1/2,W_S,\alpha_S)/L_S(1/2,\alpha\otimes \pi)]L(1/2,\alpha \otimes \pi)=0$,
 but one can always choose $W$ such that $$[\Psi(1/2,W_S,\alpha_S)/L_S(1/2,\alpha \otimes \pi)]\neq 0,$$ and $L_S(1/2,\alpha \otimes \pi)=0$. 
It is also proved in \cite{JS2} that 
the partial exterior square $L$-function $L^S(s,\Lambda^2,\pi)$ can have a pole at $1$ which is at most simple. Now the theorem follows from 
the equality $L^{lin,S}(s,\pi,\chi)=L^{S}(s,\alpha \otimes \pi)L^{S}(2s,\Lambda^2,\pi)$ and Theorem 
\ref{partialpole}.\\
Now if $Re(\alpha)<0$. If $\pi$ admits a global $\chi^{-1}$-period, then its dual representation $\tilde{\pi}$ admits 
obviously a $\chi$-period. By the previous case, we know that $L^S(s,\Lambda^2,\tilde{\pi})$ 
has a pole at $1$, and $L(1/2,\alpha^{-1} \otimes \tilde{\pi})\neq 0$. But then, by Lemma \ref{selfdual}, 
we have $\tilde{\pi}=\pi$, hence $L^S(s,\Lambda^2,\pi)$ has a pole at $1$, and we obtain 
that $L(1/2,\alpha \otimes \pi)\neq 0$ thanks to the functional equation of the Godement-Jacquet $L$-function. Conversely, if 
$L^S(s,\Lambda^2,\pi)$ has a pole at $1$, and $L(1/2,\alpha \otimes \pi)\neq 0$, by Lemma \ref{selfdual} again, 
we know that $\pi$ is equal to $\tilde{\pi}$, hence $L^S(s,\Lambda^2,\tilde{\pi})$ has a pole at $1$. Using the functional equation again, 
we also obtain $L(1/2,\alpha^{-1} \otimes \tilde{\pi})\neq 0$, hence, as $Re(\alpha^{-1})\geq 0$, we deduce that 
$\tilde{\pi}$ has a $\chi$-period, i.e. that $\pi$ has a $\chi^{-1}$-period.
\end{proof}

\section{The odd case}\label{odd}

In this section, we just state the results for the odd case, which is totally similar to the even case.\\ 

\noindent In the local non archimedean case, for $\pi$ a generic representation of $G_{n}$, 
the integrals we consider are the following for $W$ in $W(\pi,\theta)$, and $\Phi$ in $\sm_c(F^n)$:

$$\Psi(s,W,\chi,\Phi)= \int_{N_{n}\cap H_{n} \backslash H_{n}} W (h)\Phi(L_{m+1}(h_1))|h|^s\chi(h)dh.$$ 

We have:

\begin{thrm}\label{nonarchL'}
There is a real number $r_{\pi,\chi}$, such that each integral $\Psi(s,W,\chi,\Phi)$ converges for $Re(s)>r_{\pi}$. 
Moreover, when $W$ and $\Phi$ vary in $W(\pi,\theta)$ and $C_c^{\infty}(F^{n})$ respectively, they 
span a fractional ideal of $\C[q^{s},q^{-s}]$ in $\C(q^{-s})$, generated by an Euler factor which we denote 
$L^{lin}(s,\pi,\chi)$.\\
\end{thrm}

Let $\pi^0$ be a generic unramified representation of 
$G_{n}$. The unramified computation gives again, thanks to the relation $\d_{B_{n}}(a)= \d_{B_{n}}(a') \d_{B_{n}}(a'')$, with 
$a'=(a_1,a_3,\dots,a_{n})$ and $a''=(a_2,a_4,\dots,a_{n})$, the equality:

$$\Psi(s,W^0,\Phi^0) =\sum_{\l\in \Lambda^{++}}\alpha(\omega)^{c_\l}s_{\l}(q^{-s}z)=L(s,\alpha \circ \pi^0)L(s,\Lambda^2,\pi^{0})$$

In the archimedean case, for $\pi$ a generic representation of $G_{n}$, and $Re(\alpha)\geq 0$, 
the integrals $$\Psi(s,W,\chi,\Phi)= \int_{N_{n}\cap H_{n} \backslash H_{n}} W (h)\Phi(L_{m+1}(h_1))\chi(h)|h|^sdh$$ 
converge again for $Re(s)\geq 1/2-\epsilon$ for a positive number $\epsilon$ depending on $\pi$. For any such $s$, one can chose 
$W$ and $\Phi$ such that they don't vanish.\\

In the global situation, for $\Phi$ in $\mathcal{S}(\A^n)$, 
we define $$f_{\chi,\Phi}(s,h)=|h|^s\chi(h)\int_{\A^*}\Phi(aL_{m+1}(h_1))|a|^{(n)s}d^*a$$ for $h$ in $H_{n}$. 
Associated is the Eisenstein series 
$$E(s,h,\chi,\Phi)=\sum_{\gamma\in Z_{n}(k)P_{n}^\sigma(k)\backslash H_{n}(k)} f_{\chi,\Phi}(s,\gamma h),$$ which converges absolutely for 
$Re(s)>1/2$, extends meromorphically to $\C$, with possible poles simple, and located at $0$ and $1/2$. 
It satsifies the functional equation $$E(1/2-s,^t\!h^{-1},\chi^{-1}\d'^{1/2},\widehat{\Phi})=E(s,h,\chi,\Phi).$$
Then if $\pi$ is a cuspidal automorphic representation of $G_{n}(\A)$, and $\phi$ is a cusp form in the space of $\pi$, we define for $Re(s)>1/2$ 
the integral $$I(s,\phi,\chi,\Phi)=\int_{Z_{n}(\A)H_{n}(k)\backslash H_{n}(\A)} E(s,h,\chi,\Phi)\phi(h)dh,$$ 
which satisfies the statement of the following theorem.

\begin{thrm}\label{I'}
The integral $I(s,\phi,\chi,\Phi)$ extends to an entire function on $\C$. The integral $I(s,\phi,\chi,\Phi)$ 
also admits the following functional equation:
$$I(1/2-s,\tilde{\phi},\chi^{-1}\d^{1/2},\widehat{\Phi})=I(s,\phi,\chi,\Phi),$$ 
where $\tilde{\phi}:g\mapsto \phi(^t\!g^{-1})$.
\end{thrm} 
The proof is the same up to the following extra argument. It is clear that the residue of $I(s,\phi,\chi,\Phi)$ at $1/2$ is 
$$c\widehat{\Phi}(0)\int_{Z_{n}(\A)H_{n}(k)\backslash H_{n}(\A)}\chi(h) \phi(h)dh,$$ but these integrals are to known to 
vanish according to Proposition 2.1 of \cite{BF}, hence there is actually no pole at $1/2$.\\

Again, we define $$\Psi(s,W_\phi,\chi,\Phi)= \int_{N_{n}(\A)\cap H_{n}(\A) \backslash H_{n}(\A)} W_\phi (h)\Phi(L_{m+1}(h_1))|h|^s\chi(h)dh,$$ 
and this integral converges for $Re(s)$ large, and is in fact equal to $I(s,\phi,\chi,\Phi)$.\\

From this we deduce that the partial $L$-function $L^{lin,S}(s,\pi,\chi)$ is meromorphic, and, when $Re(\alpha)\geq 0$, it is holomorphic for 
$Re(s)\geq 1/2-\e$, for some positive $\e$ 
(corresponding to the $\e$ of the archimedean case).   

\bibliographystyle{plain}
\bibliography{BF}

\end{document}